\documentclass[10pt]{amsart}
\usepackage{amssymb}
\usepackage{graphicx}
\usepackage{xcolor} 
\usepackage{tensor}

\usepackage[hmargin=3cm,vmargin=3.5cm]{geometry}

\definecolor{green}{rgb}{0,0.8,0} 

\newtheorem{theorem}{Theorem}[section]

\newtheorem{lemma}[theorem]{Lemma}
\newtheorem{proposition}[theorem]{Proposition}
\theoremstyle{definition}
\newtheorem{definition}[theorem]{Definition}

\theoremstyle{remark}
\newtheorem{remark}[theorem]{Remark}
\numberwithin{equation}{section}
\newcommand{\nrm}[1]{\Vert#1\Vert}
\newcommand{\abs}[1]{\vert#1\vert}

\newcommand{\set}[1]{\{#1\}}

\newcommand{\supp}{{\mathrm{supp} \, }}

\newcommand{\lap}{\triangle}

\newcommand{\ud}{\mathrm{d}}
\newcommand{\rd}{\partial}
\newcommand{\nb}{\nabla}

\newcommand{\bb}{\Big}

\newcommand{\alp}{\alpha}

\newcommand{\dlt}{\delta}
\newcommand{\Dlt}{\Delta}
\newcommand{\eps}{\epsilon}

\newcommand{\omg}{\omega}

\newcommand{\bbN}{\mathbb N}

\newcommand{\bbR}{\mathbb R}

\newcommand{\bbT}{\mathbb T}

\newcommand{\bbZ}{\mathbb Z}

\newcommand{\calD}{\mathcal D}

\newcommand{\calF}{\mathcal F}

\newcommand{\calN}{\mathcal N}

\newcommand{\calR}{\mathcal R}
\newcommand{\calS}{\mathcal S}

\newcommand{\calU}{\mathcal U}

\newcommand{\V}{\mathrm{Vol}}
\newcommand{\Ric}{\mathrm{Ric}}
\newcommand{\Riem}{\mathrm{Riem}}
\newcommand{\Hlap}{\lap_{\mathrm{H}}}

\newcommand{\seq}{\subseteq}

\begin{document}

\title[]{A heat flow approach to Onsager's conjecture \\ for the Euler equations on manifolds}
\author{Philip Isett}%
\address{Department of Mathematics, MIT, Cambridge, MA, 02139}%
\email{isett@math.mit.edu}%

\author{Sung-Jin Oh}%
\address{Department of Mathematics, UC Berkeley, Berkeley, CA, 94720}%
\email{sjoh@math.berkeley.edu}%

\thanks{The second author is a Miller research fellow, and would like to thank the Miller institute for support. \\}%

\begin{abstract}
We give a simple proof of Onsager's conjecture concerning energy conservation for weak solutions to the Euler equations on any compact Riemannian manifold, extending the results of Constantin-E-Titi and Cheskidov-Constantin-Friedlander-Shvydkoy in the flat case. When restricted to $\mathbb{T}^{d}$ or $\mathbb{R}^{d}$, our approach yields an alternative proof of the sharp result of the latter authors. 

Our method builds on a systematic use of a smoothing operator defined via a geometric heat flow, which was considered by Milgram-Rosenbloom as a means to establish the Hodge theorem. In particular, we present a simple and geometric way to prove the key nonlinear commutator estimate, whose proof previously relied on a delicate use of convolutions.
\end{abstract}
\maketitle
\section{Introduction}
Let $(M, g_{jk})$ be a smooth $d$-dimensional complete Riemannian manifold and $I \seq \bbR$ an open interval. The \emph{incompressible Euler equations} on $I \times M$ takes the form
\begin{equation} \label{eq:euler} \tag{E}
	\left\{
\begin{aligned}
	\rd_{t} v^{\ell} + \nb_{j} (v^{j} v^{\ell}) =& - \nb^{\ell} p, \\
	\nb_{j} v^{j} =& 0.
\end{aligned}
	\right.
\end{equation}
where $v^{\ell} = v^{\ell}(t)$ is a vector field and $p = p(t)$ a function on $M$ parametrized by $t \in I$. For tensor notations, see \S \ref{subsec:notations}. 

The purpose of this note is to give a simple geometric proof of the positive direction of \emph{Onsager's conjecture} concerning energy conservation for weak solutions to the Euler equations on Riemannian manifolds. When restricted to $\bbT^{d}$ or $\bbR^{d}$, our approach gives an alternative proof of the sharp results of Constantin-E-Titi \cite{MR1298949} and Cheskidov-Constantin-Friedlander-Shvydkoy \cite{MR2422377}. 

Onsager's conjecture \cite{Onsager} states that weak solutions of the Euler equation with H\"older regularity $\alp > 1/3$ enjoy conservation of energy, whereas energy dissipation may occur when $\alp \leq 1/3$. The significance of this conjecture lies in its connection with the theory of turbulence and anomalous dissipation.  
For more discussion, we refer the reader to \cite{MR2660721} and the references therein.

The positive direction of Onsager's conjecture (i.e., energy conservation for regular weak solutions) is by now quite well-understood on flat spaces of any dimension (i.e., $\bbT^{d}$ and $\bbR^{d}$). After initial progress by Eyink \cite{Eyink1994222} following the original computations of Onsager, a beautiful and simple proof of energy conservation for velocities in the Besov space $L^{3}_{t} B^{\alp}_{3, \infty}$ with $\alp > 1/3$ was given by Constantin-E-Titi in \cite{MR1298949}, thereby confirming Onsager's conjecture in the positive direction. The criterion for energy conservation was further refined by Duchon-Robert \cite{MR1734632} and Cheskidov-Constantin-Friedlander-Shvydkoy \cite{MR2422377}. In particular, in \cite{MR2422377}, energy conservation was proven for velocities in the space $L^{3}_{t} B^{1/3}_{3, c(\bbN)}$ which is sharp in view of a negative example\footnote{
The example constructed in \cite{MR2422377} is a divergence-free vector field on $\bbR^{3}$ in $B^{1/3}_{3, \infty}$ for which the energy flux would be positive for any solution to the Euler equations obtaining this initial data. However, it has not been shown that there exist solutions with this initial data.} in the slightly larger space $B^{1/3}_{3, \infty}$ (also given in \cite{MR2422377}).

We also note the recent surge of works concerning the negative direction of Onsager's conjecture, i.e., the construction of weak solutions which dissipate or otherwise fail to conserve energy with H\"{o}lder regularity $\alp \leq 1/3$.  The first examples of 
H\"{o}lder continuous solutions to Euler with dissipating energy and H\"{o}lder exponent $\alp < 1/10$ were constructed in the setting of $\bbT^3$ by De Lellis-Sz\'{e}kelyhidi in \cite{DeLellis:2012tz} following the construction of continuous solutions 
with dissipating energy in \cite{DeLellis:2012uza}.  These results were generalized to the setting of $\bbT^2$ by Choffrut-De Lellis-Sz\'{e}kelyhidi in \cite{deLSzeCts2d} and by Choffrut in \cite{choff} where a description of weak limits of these solutions was also obtained.  These solutions were constructed by the method of convex integration, a weaker form of which was first developed for the construction of $L^\infty$ solutions to the Euler equations by De Lellis-Sz\'{e}kelyhidi in 
\cite{DeLellis:2009jh}.  Solutions with H\"{o}lder regularity $\alp< 1/5$ which fail to conserve energy were later constructed by Isett in \cite{Isett:2012ub} building on ideas of \cite{DeLellis:2012uza, DeLellis:2012tz}  and introducing further improvements in the framework.  Another proof of the main result of \cite{Isett:2012ub} was given by Buckmaster-De Lellis-Sz\'{e}kelyhidi in 
\cite{Buckmaster:2013vy}, including the construction of Euler flows with H\"{o}lder regularity $\alp < 1/5$ and decreasing energy profiles.  Buckmaster \cite{Buckmaster:2013vv} has also shown that the solutions can be made to be $(1/3 -\epsilon)$-H\"{o}lder in space at almost every time through modifications in the construction.  Nonetheless, the current method appears to be limited to 
the regularity $\alp < 1/5$ and the negative direction of Onsager's conjecture remains an open problem.


Previous approaches to the positive direction of Onsager's conjecture employed convolutions and Fourier-analytic techniques (more precisely, Littlewood-Paley theory), both of which rely heavily on the translational symmetry of $\bbT^{d}$ and $\bbR^{d}$. 
Consideration of this problem on general manifolds necessitates a more geometric way of regularizing and measuring smoothness of the weak solutions to the Euler equations, and also of effectively exploiting their nonlinear structure.

Our starting point is to define a smoothing operator using a \emph{geometric heat flow} instead of convolution, as the latter is not available in general. 
More precisely, given a vector field $u^{\ell}$, we let $\calS[s] u^{\ell}$ be the solution to a geometric heat equation (with heat-time $s \geq 0$) such that $\calS[0] u^{\ell} = u^{\ell}$. 
When written for 1-forms via the metric duality, the heat flow we employ takes the form
\begin{equation*} \tag{HHF}
	\rd_{s} \omg - \Hlap \omg = 0,
\end{equation*}
where the operator $\Hlap := -(\dlt \ud + \ud \dlt)$ is the standard \emph{Hodge Laplacian} on forms. Accordingly, we shall henceforth refer to the above as the \emph{Hodge heat flow} \eqref{eq:HHF}. 
This was considered earlier by Milgram-Rosenbloom \cite{MR0042769} as an alternative means to establish the celebrated Hodge theorem on compact manifolds. 

Our proof of the positive direction of Onsager's conjecture on manifolds relies on the following ideas:
\begin{enumerate}
\item We observe that the divergence-free property (i.e., $\nb_{\ell} v^{\ell} = 0$) is invariant under \eqref{eq:HHF}, i.e., the evolution of a divergence-free vector field via \eqref{eq:HHF} remains divergence-free. Therefore, \eqref{eq:HHF} provides a natural and geometric way to smooth out weak solutions to the Euler equations.
\item In order to get an analogue of the sharp result of \cite{MR2422377}, we need a means to measure smoothness of divergence-free vector fields in the Besov sense. Motivated by the use of  heat flows in the formulation of Littlewood-Paley theory on manifolds in \cite{MR0252961, MR2221254}, we use \eqref{eq:HHF} to directly \emph{define} a scale of Besov-type spaces we need. Our definitions coincide with the standard Littlewood-Paley theoretic definitions when $M = \bbT^{d}$ or $\bbR^{d}$, and satisfy the usual embedding properties.
\item The nonlinear commutator term which arises when we apply our smoothing operator to the Euler equations exhibits cancellations analogous to those in \cite{MR1298949, MR2422377}. This is the basis of our proof of the nonlinear commutator estimate (Theorem \ref{thm:commEst}), which lies at the heart of the whole proof.
\end{enumerate}

Our method may be compared with those in \cite[Eq. (10)]{MR1298949} and \cite[Eq. (20)]{MR2422377}, both of which involve a delicate use of convolutions on $\bbT^{d}$ and $\bbR^{d}$.  
It turns out to be remarkably easy to reveal the necessary cancellations of the nonlinear commutator in our framework. One simply derives a parabolic PDE by applying $\rd_{s} - \Hlap$ to the commutator and then uses Duhamel's principle, exploiting the fact that the commutator is zero at $s=0$ in the sense of distributions (see Section \ref{sec:commEst} for more details). 

\subsection{Statement of the main results}
We now give precise statements of our main results. In the rest of this note, the Riemannian manifold $M$ will be always assumed to be smooth.

We begin with a definition of a \emph{weak solution} to the Euler equations.
\begin{definition}[Weak solution to Euler] \label{def:weakSol2Euler}
Let $I \seq \bbR$ be an open interval. We say that $(v^{\ell}, p) \in L^{2}_{\mathrm{loc}} (I \times M) \times L^{1}_{\mathrm{loc}} (I \times M)$ is a \emph{weak solution to the Euler equations} on $I$ if $v^{\ell}$ is divergence-free (i.e., $\nb_{\ell} v^{\ell} = 0$ in the sense of distributions) and for every $\omg_{\ell} \in C_{c}^{\infty}(I \times M)$ we have
\begin{equation} \label{eq:weakEuler}
	\iint_{I \times M}  v^{\ell} \rd_{t} \omg_{\ell} + v^{j} v^{\ell} \nb_{j} \omg_{\ell} +  p \, \nb^{\ell} \omg_{\ell} \, \ud^{1+d} \V = 0.
\end{equation}
\end{definition}

For simplicity, we shall henceforth limit ourselves to smooth \emph{compact} Riemannian manifolds. We remark, though, that all \emph{quantitative} estimates used below hold under much less stringent assumptions, i.e., they hold on smooth complete Riemannian manifolds with sectional curvatures uniformly bounded from above and below\footnote{An important point is that we only rely on \emph{short-time} parabolic estimates, which are much more robust than their long-time counterparts.}. 

%

We are now ready to state our main theorems. We begin with a simpler version, which is stated in terms of the standard H\"older and Sobolev spaces. 
%

\begin{theorem}[Onsager's conjecture on manifolds, simple version] \label{thm:Onsager}
Let $(M, g_{jk})$ be a compact Riemannian manifold and $I \seq \bbR$ an open interval.
Let $(v^{\ell}, p)$ be a weak solution to the Euler equations on $I \times M$ such that $v^{\ell} \in C_{t}(I; L^{2}(M))$. Then the following statements hold.
\begin{enumerate}
\item If $v^{\ell} \in L^{3}_{t}(I; C^{\alp} (M))$ with $\alp > 1/3$, then then conservation of energy holds, i.e., for all $t_{1}, t_{2} \in I$ we have
\begin{equation} \label{eq:energy}
	\frac{1}{2} \int_{M} \abs{v(t_{1})}^{2} \, \ud^{d} \V
	= \frac{1}{2} \int_{M} \abs{v(t_{2})}^{2} \, \ud^{d} \V.
\end{equation}

\item If $v^{\ell} \in L^{3}_{t}(I; W^{\alp, 3} (M))$ with $\alp \geq 1/3$, then then conservation of energy holds.
\end{enumerate}
\end{theorem}

The H\"older space $C^{\alp}(M)$ in the theorem is defined by
\begin{equation*}
	C^{\alp}(M) := \set{u^{\ell} \in C^{0}(M) : \nrm{u^{\ell}}_{C^{0}(M)} + \sup_{x \in M} \sup_{h \in T_{x} M} \abs{h}^{-\alp} \abs{\dlt_{x} u(h)} < \infty},
\end{equation*}
where $\dlt_{x} u(h) \in T_{x} M$ is the difference between $u^{\ell}(x)$ and the parallel transport of $u^{\ell}(\exp_{x}(h))$ to $x = \exp_{x}(0)$ along the radial geodesic, where $\exp_{x}$ is the exponential map at $x$. The fractional Sobolev space $W^{\alp, 3}$(M) $(0 < \alp < 1)$ is defined to be the complex interpolation space between $L^{3}(M)$ and $W^{1,3}(M)$. For a general discussion on complex interpolation, we refer the reader to \cite{MR0482275}. 

In fact, our proof gives a sharper but more technical criterion, which reduces to that of Cheskidov-Constantin-Friedlander-Shvydkoy \cite[Theorem 3.3]{MR2422377}\footnote{The theorem in \cite[Theorem 3.3]{MR2422377} is stated on $\bbR^{d}$, but the same proof applies to the $\bbT^{d}$ case. On the other hand, our result is stated only for compact manifolds, but the proof also applies to the case $\bbR^{d}$.} when restricted to $M = \bbT^{d}$. Thus, our approach furnishes an alternative proof of the sharp result in \cite{MR2422377}. To state this version, we need a means to measure smoothness of vector fields in the Besov sense. As discussed in the introduction, we rely on the Hodge heat flow \eqref{eq:HHF} for this purpose as follows.

%

\begin{definition} \label{def:besov}
Let $(M, g_{jk})$ be a compact Riemannian manifold and $I \seq \bbR$ an open interval. Let $e^{s \Hlap}$ be the heat semi-group generated by the Hodge Laplacian $\Hlap$ (see Section \ref{sec:HHF} for more on this semi-group). Given $0 < \alp < 1$ and $1 \leq p < \infty$, we define the space $B^{\alp}_{p, c(\bbN)}$ of vector fields to be the completion of $C^{\infty}_{c}(M)$ with respect to  the norm
\begin{equation} \label{eq:besovInfty}
	\nrm{u^{\ell}}_{B^{\alp}_{p, \infty}(M)} := \nrm{u^{\ell}}_{L^{p}(M)} + \sup_{s \in (0,1]} s^{\frac{1-\alp}{2}} \nrm{\nb e^{s \Hlap} u^{\ell}}_{L^{p}(M)}.
\end{equation}
\end{definition}

The preceding definition is justified by the fact that on $M = \bbT^{d}$ or $\bbR^{d}$, it coincides with the standard Littlewood-Paley theoretic definition as given in \cite{MR2422377}. Indeed, on a flat space, $e^{s \Hlap}$ is simply the (component-wise) standard heat semi-group $e^{s \lap}$, and \eqref{eq:besovInfty} becomes a well-known characterization of the (inhomogeneous) Besov norm via the heat semi-group. For completeness, we sketch a proof of this fact in Appendix \ref{sec:equivBesov}. 

\begin{remark} 
For $0 < \alp < 1$ and $1 \leq p < \infty$, the natural definition for the space $B^{\alp}_{p, \infty}(M)$ is
\begin{equation*} 
	B^{\alp}_{p, \infty}(M) 
	:= \set{ u^{\ell} \in \calD'(M): \nrm{u^{\ell}}_{B^{\alp}_{p, \infty}(M)} < \infty},
\end{equation*}
where $\calD'(M)$ is the space of distributions on $M$ (i.e., dual space of smooth compactly supported 1-forms). 
Indeed, when $M = \bbT^{d}$ or $\bbR^{d}$, this also coincides with the standard definition. We remark that $B^{\alp}_{p, c(\bbN)}(M)$ is a proper subset of $B^{\alp}_{p, \infty}(M)$; for example, see Lemma \ref{lem:cNBesov}. 
\end{remark}

\begin{remark} 
It is also possible to define Besov spaces with more general `summability' exponent. 
Indeed, given $0 < \alp < 1$ and $1 \leq p, r < \infty$, we define the space $B^{\alp}_{p, r}(M)$ of vector fields to be the completion of $C^{\infty}_{c}(M)$ with respect to the norm
\begin{equation*}
	\nrm{u^{\ell}}_{B^{\alp}_{p, r}(M)} := \nrm{u^{\ell}}_{L^{p}(M)} + \nrm{s^{\frac{1-\alp}{2}} \nrm{\nb e^{s \Hlap} u^{\ell}}_{L^{p}(M)}}_{L^{r}((0, 1], \ud s / s)}.
\end{equation*}

Again, these spaces coincide with the standard definition on $M = \bbT^{d}$ or $\bbR^{d}$ (see Appendix \ref{sec:equivBesov}). Moreover, they satisfy the standard embedding property
\begin{equation*}
	B^{\alp}_{p, r} \seq B^{\alp}_{p, r'} \seq B^{\alp}_{p, c(\bbN)} \quad \hbox{ for } \quad 
	1 \leq r \leq r' < \infty.
\end{equation*}
\end{remark}

The spaces $C^{\alp}(M)$, $W^{\alp, p}(M)$ and $B^{\alp}_{p, c(\bbN)}(M)$ are related to each other as follows.
\begin{proposition} \label{prop:embed}
Let $M$ be a smooth compact Riemannian manifold. Then the following statements hold.
\begin{enumerate}
\item For $0 < \alp \leq 1$ and $2 \leq p < \infty$, we have $W^{\alp, p}(M) \seq B^{\alp}_{p, c(\bbN)}(M)$.
\item For $0 < \alp < \alp' \leq 1$ and $1 \leq p < \infty$, we have $C^{\alp'}(M) \seq W^{\alp, p}(M)$.
\end{enumerate}
\end{proposition}

We remark that in (2), the strict inequality $\alp < \alp'$ is necessary, as one can readily verify in the case $M = \bbT^{d}$. Taking $M = \bbR^{d}$, it can also be seen that this inclusion is false without the assumption of compactness. A proof of this proposition will be given at the end of Appendix \ref{sec:heatFlowEst}. 

We are now ready to state the sharper version of our main theorem.

\begin{theorem}[Onsager's conjecture on manifolds, sharp version] \label{thm:Onsager:difficult}
Let $(M, g_{jk})$ be a compact Riemannian manifold and $I \seq \bbR$ an open interval.
Let $(v^{\ell}, p)$ be a weak solution to the Euler equations on $I \times M$ such that $v^{\ell} \in L^{3}_{t}(I; B^{1/3}_{3, c(\bbN)}(M)) \cap C_{t} (I; L^{2}(M))$. Then conservation of energy \eqref{eq:energy} holds.
\end{theorem}

\begin{remark} 
Theorem \ref{thm:Onsager} is an immediate corollary of Proposition \ref{prop:embed} and Theorem \ref{thm:Onsager:difficult}. 
\end{remark}

\begin{remark} 
As we shall show in Appendix \ref{sec:equivBesov}, $B^{1/3}_{3, c(\bbN)}(M)$ coincides with the standard definition as in \cite{MR2422377} when $M = \bbT^{d}$ or $\bbR^{d}$. Thus Theorem \ref{thm:Onsager:difficult} directly gives the sharp criterion of \cite{MR2422377} on $\bbT^{d}$. Moreover, the methods of this note apply to the non-compact case of $\bbR^{d}$, if we enlarge the set of test functions in Definition \ref{def:weakSol2Euler} to the space $\calS(\bbR^{d})$ of Schwartz 1-forms as in \cite{MR2422377}. This recovers \cite[Theorem 3.3]{MR2422377}.
\end{remark}

\subsection{Outline of the note}
The rest of this note is concerned mainly with the proof of Theorem \ref{thm:Onsager:difficult}. In Section \ref{sec:abstractOnsager}, we give an abstract theorem concerning energy conservation, which reduces the proof of Theorem \ref{thm:Onsager:difficult} to finding an appropriate smoothing operator $\calS[s]$ and a function space $X$ on which the crucial \emph{commutator estimate} \eqref{eq:babyCommEst} holds. In Section \ref{sec:HHF}, we construct a smoothing operator $\calS[s]$ in terms of \eqref{eq:HHF}, i.e., $\calS[s] := e^{s \Hlap}$. Here we show, in particular, the invariance of divergence-free property under \eqref{eq:HHF}. Then in Section \ref{sec:commEst}, we establish the commutator estimate \eqref{eq:babyCommEst} with $X = L^{3}_{t}(I; B^{1/3}_{3, c(\bbN)}(M))$, thereby completing the proof of Theorem \ref{thm:Onsager:difficult}. 

Our note is complemented with two appendices. In Appendix \ref{sec:heatFlowEst}, we sketch a simple approach to prove short-time $L^{p}$ estimates for \eqref{eq:HHF} on a general class of manifolds; this will be used to give a proof of Proposition \ref{prop:embed}. In Appendix \ref{sec:equivBesov}, we establish the equivalence between Definition \ref{def:besov} in the case $M = \bbT^{d}$ or $\bbR^{d}$ with the standard definition of Besov spaces via Littlewood-Paley theory. 

\subsection{Notations and conventions} \label{subsec:notations}
\begin{itemize}
\item We use $\bbR$ and $\bbN$ to denote the real line and non-negative integers, respectively.
\item We employ the abstract index notation, by which we mean the indices are used as place-holders indicating the type of a tensor and which components are contracted. According to this notation, a vector field $v^{\ell}$ on $M$ has an upper index, whereas a 1-form $\omg_{\ell}$ on $M$ has a lower index.
Indices are raised or lowered using the metric $g_{jk}$ and repeated upper and lower indices are summed up.

\item The covariant derivative on $M$ is denoted by $\nb_{j}$. For $L \in \bbN$, an $L$-fold covariant derivative will be denoted by $\nb^{(L)}_{j_{1} \cdots j_{L}}$. 
	 The Riemann curvature tensor is denoted by $\Riem_{ijk\ell}$, which is defined by the relation $(\nb^{(2)}_{ij}  - \nb^{(2)}_{ji}) u^{\ell} = \tensor{\Riem}{_{ij}^{\ell}_{k}} u^{k}$. The Ricci curvature is denoted by 
$\Ric_{ij} := \tensor{\Riem}{_{ki}^{k}_{j}}$.

\item For a tensor $T$, we write $\abs{T}^{2}$ for the induced inner product of $T$ with itself. The $d$- and $(d+1)$-dimensional induced volume on $M$ and $I \times M$ are denoted by $\ud^{d} \V$ and $\ud^{1+d} \V = \ud t \, \ud^{d} \V$, respectively. 

\item The space $L^{p}(M)$ is defined using the above norm and measure. The Sobolev space $W^{1, p}(M)$ is defined with the norm $\nrm{u}_{W^{1,p}(M)} :=\nrm{u}_{L^{p}(M)}+ \nrm{\nb u}_{L^{p}(M)}$. The space $C^{0}(M)$ consists of all bounded continuous tensors, and $C^{\infty}_{c}(M)$ is the space of all smooth compactly supported tensors (both of a given type).
\end{itemize}

\section{Abstract theorem on conservation of energy} \label{sec:abstractOnsager}
Let $(v^{\ell}, p)$ be a weak solution to the Euler equations. To prove conservation of energy \eqref{eq:energy}, it suffices to show that for any smooth function $\eta(t)$ supported in $I$, we have
\begin{equation} \label{eq:energy2}
	- \iint_{\bbR \times M} \eta'(t) \frac{\abs{v(t)}^{2}}{2} \, \ud^{1+d} \V = 0.
\end{equation}

Indeed, as $v^{j} \in C_{t} (I; L^{2}_{x})$, we may take $\eta$ to be the characteristic function of a time interval $(t_{1}, t_{2}) \seq I$ by approximation.

For a smooth solution to the Euler equations, \eqref{eq:energy2} is proven by multiplying the equation with the test function $\eta(t) v^{\ell}(t)$ and integrating by parts. However, such a procedure is not justified for a weak solution. One may nevertheless attempt to carry out this proof by first approximating the weak solution by smooth solutions and then handling the error from approximation. This motivates the following definition.

\begin{definition} \label{def:smthOp}
We say that an operator $\calS[s]$ (parametrized by $s \in (0,1]$) on the space of $L^{2}(M)$ vector fields is a \emph{smoothing operator} if the following properties hold:
\begin{enumerate}
\item For each $s \in (0, 1]$, $\calS[s]$ is a bounded self-adjoint operator on $L^{2}(M)$.
\item For $u^{\ell} \in L^{2}(M)$, $\calS[s] u^{\ell} \to u^{\ell}$ in $L^{2}(M)$ as $s \to 0$.
\item For $u^{\ell} \in L^{2}(M)$, $L \in \bbN$ and $s \in (0,1]$, we have $\nrm{\nb^{(L)}_{j_{1} \cdots j_{L}} \calS[s] u}_{L^{2}(M) \cap C^{0}(M)} \leq C_{s, L} \nrm{u}_{L^{2}(M)}$.
\end{enumerate}

Using the metric, we extend $\calS[s]$ to $L^{2}(M)$ one-forms as well by lowering the index.
\end{definition}

The following theorem reduces the problem of establishing conservation of energy to that of finding a smoothing operator which satisfies certain additional properties.
\begin{theorem} \label{thm:abstractOnsager}
Let $(M, g_{jk})$ be a compact Riemannian manifold, $I \seq \bbR$ an open interval. Let $(v^{\ell}, p)$ be a weak solution to the Euler equations on $I \times M$ such that $v^{\ell} \in C_{t}(I; L^{2}(M))$. Let $\calS[s]$ be a smoothing operator which satisfies the following additional properties:
\begin{enumerate}
\item (Invariance of divergence) For every $u^{\ell} \in L^{2}(M)$, $\calS[s] u^{\ell}$ is divergence-free if $u^{\ell}$ is.
\item (Commutator estimate) There exists a set $X$ of vector fields on $I \times M$ such that for every $u^{\ell} \in X$ and smooth function $\eta(t)$ with compact support on $I$, we have\footnote{The expression $\iint \eta(t) \calS[s] \nb_{j}(u^{j} u^{\ell}) \calS[s] u_{\ell} \, \ud^{1+d} \V$ is to be interpreted in the weak sense, i.e., as $- \iint \eta(t) (u^{j} u^{\ell}) \nb_{j} \calS[s]^{2} u_{\ell} \, \ud^{1+d} \V$.}
\begin{equation} \label{eq:babyCommEst}
	\bb\vert \iint_{I \times M} \eta(t) \bb( \calS[s] \nb_{j}(u^{j} u^{\ell}) - \nb_{j}(\calS[s] u^{j} \calS[s] u^{\ell}) \bb) \calS[s] u_{\ell} \, \ud^{1+d} \V  \bb\vert
	\to 0 \hbox{ as } s \to 0.
\end{equation}
\end{enumerate}

Then under the additional assumption $v^{\ell} \in X$, conservation of energy \eqref{eq:energy} holds.
\end{theorem}

\begin{remark} 
According to Theorem \ref{thm:abstractOnsager}, the previous proofs on $\bbT^{d}$ or $\bbR^{d}$ may be viewed as choosing $\calS[s]$ to be the standard convolution mollifier $s^{-d} \varphi(\cdot /s) \ast $ and $X = L^{3}_{t}(I; B^{\alp}_{3, \infty})$ with $\alp > 1/3$ (see \cite{MR1298949}) or $L^{3}_{t}(I, B^{1/3}_{3, c(\bbN)})$ (see \cite{MR2422377}). 
As discussed in the Introduction, such approaches rely heavily on the translational symmetry of these spaces and are therefore difficult to generalize to general manifolds. 
Instead, we shall take $\calS[s] = e^{s \Hlap}$ (see Section \ref{sec:HHF}) and $X = L^{3}_{t} (I; B^{1/3}_{3, c(\bbN)}(M))$ as in Definition \ref{def:besov}.
\end{remark}

\begin{proof} 
The space-time integrals below are all taken over $\bbR \times M$.
We denote by $[v]_{\dlt} := \varphi_{\dlt} \ast_{t} v$ a standard mollification in the time-variable, where $\varphi$ is a smooth compactly supported function on $\bbR$ with $\int \varphi(t) \, \ud t= 1$. 

As discussed above, it suffices to establish \eqref{eq:energy2}.
Using the fact $v^{\ell} \in C_{t} (I; L^{2}(M))$ and properties of $\calS[s]$, we have
\begin{equation*}
	- \iint \eta'(t) \frac{\abs{v}^{2}}{2} \, \ud^{1+d} \V
	= - \lim_{s \to 0} \lim_{\dlt \to 0} \iint \eta'(t) \frac{\abs{\calS[s] [v]_{\dlt}}^{2}}{2} \, \ud^{1+d} \V
\end{equation*}

Using the self-adjointness of $\calS[s]$ and $[\cdot]_{\dlt}$, it is not difficult to prove
\begin{equation*}
- \iint \eta'(t) \frac{\abs{\calS[s] [v]_{\dlt}}^{2}}{2} \, \ud^{1+d} \V
= - \iint  v^{\ell} \rd_{t} [\eta(t) \calS[s]^{2} [v_{\ell}]_{\dlt}]_{\dlt} \, \ud^{1+d} \V
\end{equation*}

Note that $[\eta(t) \calS[s]^{2} [v_{\ell}]_{\dlt}]_{\dlt}$ is a valid test function\footnote{In the case of $\bbR^{d}$, when $\calS[s]$ is the standard heat semi-group, we need an additional approximation of $v^{\ell}$ by divergence-free Schwartz vector fields to ensure that $[\eta(t) \calS[s]^{2} [v_{\ell}]_{\dlt}]_{\dlt}$ is a valid Schwartz test function.}. Thus, we may use the (weak formulation of the) Euler equations to deduce
\begin{align*}
&- \iint \eta'(t) \frac{\abs{\calS[s] [v]_{\dlt}}^{2}}{2} \, \ud^{1+d} \V \\
& \quad = \iint v^{j} v^{\ell} \nb_{j} [\eta(t) \calS[s]^{2} [v_{\ell}]_{\dlt}]_{\dlt} \, \ud^{1+d} \V
+ \iint p \nb^{\ell} [\eta(t) \calS[s]^{2} [v_{\ell}]_{\dlt}]_{\dlt} \, \ud^{1+d} \V.
\end{align*}

Thanks to Property (1) (invariance of divergence), the second integral vanishes. 
Also, using Property (3) of Definition \ref{def:smthOp} and $v^{\ell} \in C_{t} (I; L^{2}(M))$, we can take $\dlt \to 0$ at this point and the first integral converges to 
\begin{equation*}
\iint \eta(t) v^{j} v^{\ell} \nb_{j} \calS[s]^{2} v_{\ell} \, \ud^{1+d} \V
\end{equation*} 


On the other hand, again by Property (1) (invariance of divergence), we have
\begin{equation*}
\iint \eta(t) \calS[s] v^{j} \calS[s] v^{\ell} \nb_{j} \calS[s] v_{\ell} \, \ud^{1+d} \V
= \frac{1}{2} \iint \eta(t) \calS[s] v^{j} \nb_{j} (\calS[s] v^{\ell}  \calS[s] v_{\ell}) \, \ud^{1+d} \V
= 0.
\end{equation*}

Subtracting the above quantity, we finally obtain
\begin{equation*}
- \iint \eta'(t) \frac{\abs{v}^{2}}{2} \, \ud^{1+d} \V
= \lim_{s \to 0} \iint \eta(t) \bb( v^{j} v^{\ell} \nb_{j} \calS[s]^{2} v_{\ell} - \calS[s] v^{j} \calS[s] v^{\ell} \nb_{j} \calS[s] v_{\ell} \bb) \, \ud^{1+d} \V.
\end{equation*}

Now applying Property (2) (commutator estimate), we conclude \eqref{eq:energy2}. \qedhere
\end{proof}

\section{Construction of $\mathcal{S}[s]$ via the Hodge heat flow} \label{sec:HHF}
Consider the Hodge Laplacian on 1-forms, which is defined by
\begin{equation} \label{eq:Hlap:form}
\Hlap \omg := - (\ud \dlt + \dlt \ud ) \omg_{\ell},
\end{equation}
where $\ud$, $\dlt$ are the exterior differential and co-differential operators, respectively. Note that \eqref{eq:Hlap:form} also defines the Hodge Laplacian for any $k$-form, and in particular for scalar functions, which are $0$-forms.

By the Weitzenb\"ock formula, $\Hlap$ takes the following tensorial form:
\begin{equation*} 
	\Hlap \omg_{\ell} := \nb_{j} \nb^{j} \omg_{\ell} - \tensor{\Ric}{^{k}_{\ell}} \omg_{k}.
\end{equation*}

From this expression, we can read off its formal adjoint for vector fields on $M$, which we shall denote also by $\Hlap$:
\begin{equation} \label{eq:Hlap:vf}
	\Hlap u^{\ell} := \nb_{j} \nb^{j} u^{\ell} - \tensor{\Ric}{^{\ell}_{k}} u^{k}.
\end{equation}

In fact, \eqref{eq:Hlap:form} and \eqref{eq:Hlap:vf} are just different descriptions of the same object, as they are conjugate to each other by index raising/lowering (i.e., metric duality). It turns out that \eqref{eq:Hlap:form} is useful for understanding delicate structural properties (see e.g. Proposition \ref{prop:invDiv}), whereas \eqref{eq:Hlap:vf} is more convenient for doing estimates (see Theorem \ref{thm:commEst}).

On a complete Riemanian manifold, it is well-known that $\ud \dlt + \dlt \ud$ is self-adjoint with respect to the $L^{2}(M)$ bilinear product on forms (defined using $g_{jk}$); see \cite[Theorem 2.4]{Strichartz:1983uw}. It therefore follows that $\lap_{H}$ as in \eqref{eq:Hlap:vf} is self-adjoint on the space of vector fields in $L^{2}(M)$ as well. Using $L^{2}(M)$ spectral theory, one can construct the heat semi-group $e^{s \Hlap}$ associated to the Hodge heat flow
\begin{equation} \label{eq:HHF} \tag{HHF}
	(\rd_{s} - \Hlap) U^{\ell} = 0.
\end{equation}

Some basic properties of $e^{s \Hlap}$ are in order.
\begin{lemma} \label{lem:props4heatFlow}
Let $u^{\ell} \in L^{2}(M)$. Then the following statements hold.
\begin{enumerate}
\item For every $s_{0} > 0$, $e^{s \Hlap} u^\ell$ is the unique solution to \eqref{eq:HHF} on $M \times (0, s_{0})$ such that $e^{s \Hlap} \in C_{s}((0, s_{0}); L^{2}(M))$ and $\lim_{s \to 0} e^{s \Hlap} u^\ell = u^{\ell}$ in $L^{2}(M)$. 
\item For every $s >0$, we have $\nrm{e^{s \Hlap} u}_{L^{2}(M)} \leq \nrm{u}_{L^{2}(M)}$.
\item $e^{s \Hlap}$ is a semi-group, i.e., for every $s_{1}, s_{2} > 0$, $e^{(s_{1} + s_{2}) \Hlap} = e^{s_{1} \Hlap} e^{s_{2} \Hlap}$.
\item For $s \in (0, 1]$ and $L \in \bbZ$, we have
\begin{equation} \label{eq:L2smth}
	\nrm{\nb^{(L)}_{j_{1} \cdots j_{L}} (e^{s \Hlap} u^\ell)}_{L^{2}(M)} \leq C_{s, L} \nrm{u}_{L^{2}(M)}.
\end{equation}

The constant $C_{s, L}$ is uniform on any compact subset of $(0, 1]$. Moreover, by Sobolev, it follows that $\nb^{(L)}_{j_{1} \cdots j_{L}} (e^{s \Hlap} u^\ell) \in C^{0}(M)$ as well.

\end{enumerate}
\end{lemma}
\begin{proof} 
All statements except the last can be read off from \cite[\S 3]{Strichartz:1983uw}.
The last statement follows by a standard energy integral method\footnote{We remark that the case $L=1$ of \eqref{eq:L2smth} is established in Appendix \ref{sec:heatFlowEst}, assuming only $L^{\infty}(M)$ bounds on the curvature.  The same argument can be extended to the case $L \geq 2$ of \eqref{eq:L2smth} by commuting more derivatives into \eqref{eq:HHF}.} and Sobolev. 
\qedhere
\end{proof}

In view of the preceding lemma, we choose our smoothing operator to be $e^{s \Hlap}$, i.e.,
\begin{equation}
	\calS[s] := e^{s \Hlap} \hbox{ for } s \in (0,1].
\end{equation}

A remarkable property of $\Hlap$, also inherited by $e^{s \Hlap}$, is that it commutes with the divergence operator.  This is most easily seen using \eqref{eq:Hlap:form}.
\begin{proposition}[Invariance of divergence] \label{prop:invDiv}
The following statements hold.
\begin{enumerate}
\item For every $\omg_{\ell} \in L^{2}(M)$, $\dlt \Hlap \omg = \Hlap \dlt \omg$ and $\dlt e^{s\Hlap} \omg = e^{s\Hlap} \dlt \omg$ for every $s \geq 0$.
\item For every $u^{\ell} \in L^{2}(M)$, $\nb_{\ell} \Hlap u^{\ell} = \Hlap \nb_{\ell} u^{\ell}$ and $\nb_{\ell} e^{s\Hlap} u^{\ell} = e^{s\Hlap} \nb_{\ell} u^{\ell}$ for every $s \geq 0$. 
\end{enumerate}

In particular, $\calS[s] = e^{s \Hlap}$ satisfies Property (1) of Theorem \ref{thm:abstractOnsager}.
\end{proposition}
\begin{proof} 
The second statement is equivalent to the first by index raising/lowering, and the fact that $\dlt$ conjugates to the divergence operator for vector fields, i.e., $\dlt \omg = - \nb^{\ell} \omg_{\ell}$. Given $\omg_{\ell} \in C^{\infty}_{c}(M)$, the identity $\dlt \Hlap \omg = \Hlap \dlt \omg$ is obvious from \eqref{eq:Hlap:form} and $\dlt^{2} = 0$. This identity is then extended to the operator $e^{s \Hlap}$, and for general $\omg \in L^{2}(M)$ by approximation (in the sense of distributions). \qedhere
\end{proof}

\section{Proof of the commutator estimate \eqref{eq:babyCommEst}} \label{sec:commEst}
The main result of this section is
%
\begin{theorem}[Commutator estimate] \label{thm:commEst}
Let $(M, g_{jk})$ be a compact Riemannian manifold, $I \seq \bbR$ an open interval and $\alp \geq 1/3$. Then for every $u^{\ell} \in L^{3}_{t}(I; B^{\alp}_{3, c(\bbN)}(M))$ and smooth function $\eta(t)$ supported on $I$, we have
\begin{equation} \label{eq:commEst}
\begin{aligned}
	&\bb\vert \iint_{I \times M} \eta(t) \bb( e^{s \Hlap} \nb_{j}(u^{j} u^{\ell}) - \nb_{j} \big( e^{s \Hlap} u^{j} \, e^{s \Hlap} u^{\ell} \big) \bb) e^{s \Hlap} u_{\ell} \, \ud^{1+d} \V \bb\vert 
	\to 0 \hbox{ as } s \to 0.
\end{aligned}
\end{equation}
\end{theorem}

This theorem immediately implies that $\calS[s] = e^{s \Hlap}$ satisfies Property (2) of Theorem \ref{thm:abstractOnsager} (commutator estimate), with $X = L^{3}_{t}(I; B^{1/3}_{3, c(\bbN)}(M))$. Combined with Proposition \ref{prop:invDiv}, this concludes the proof of Theorem \ref{thm:Onsager:difficult}.

For $t \in I$ and $s \in (0,1]$, we define the commutator
\begin{equation*}
W^{\ell}(t, s) := e^{s \Hlap} \nb_{j}(u^{j}(t) u^{\ell}(t)) - \nb_{j}(e^{s \Hlap} u^{j}(t) e^{s \Hlap} u^{\ell}(t)).
\end{equation*}

In the rest of this section, we shall often omit writing $t \in I$ and use the shorthand $U^{\ell}(s) := e^{s \Hlap} u^{\ell}$. 

It is easy to show that $W^{\ell}(s)$ converges to $0$ as $s \to 0$ in the sense of distributions; Theorem \ref{thm:commEst} upgrades this to a stronger statement \eqref{eq:commEst}. Our key idea is to derive a parabolic PDE for $W^{\ell}$ by applying $\rd_{s} - \Hlap$, and then to use Duhamel's principle. An interesting fact is that the seemingly naive act of computing $\rd_{s} - \Hlap$ of $W^{\ell}$ already reveals a structure analogous to that behind the delicate commutator estimate of Constantin-E-Titi \cite[Eq. (10)]{MR1298949} (see also \cite[Eq. (20)]{MR2422377}). 

The following simple lemma is crucial for getting the desired vanishing of $W^{\ell}(s)$ as $s \to 0$. It may also be interpreted as making precise the difference between $B^{\alp}_{p, \infty}(M)$ and $B^{\alp}_{p, c(\bbN)}(M)$. 
\begin{lemma} \label{lem:cNBesov}
Let $0 < \alp < 1$ and $2 \leq p < \infty$. Then for $u^{\ell} \in B^{\alp}_{p, c(\bbN)}(M)$, we have
\begin{equation} \label{eq:cNBesov}
	s^{\frac{1-\alp}{2}} \nrm{\nb U(s)}_{L^{p}(M)} \to 0 \hbox{ as } s \to 0.
\end{equation}
\end{lemma} 
\begin{proof} 
This lemma follows from the following estimate, whose proof will be provided in Appendix \ref{sec:heatFlowEst}: For $2 \leq p < \infty$, $0 < s \leq 1$ and $u^{\ell} \in C^{\infty}_{c}(M)$, we have
\begin{equation*} \tag{\ref{eq:heatFlowEst:Lp:2}}
\nrm{\nb U(s)}_{L^{p}(M)} \leq C (\nrm{\nb u}_{L^{p}(M)} + \nrm{u}_{L^{p}(M)} ).
\end{equation*}

Indeed, note that the statement \eqref{eq:cNBesov} is preserved under taking limits with respect to the $B^{\alp}_{p, \infty}(M)$ norm; thus, it suffices to prove \eqref{eq:cNBesov} on a dense subset of $B^{\alp}_{p, c(\bbN)}(M)$. By \eqref{eq:heatFlowEst:Lp:2}, the statement \eqref{eq:cNBesov} is trivial for $u^{\ell} \in C^{\infty}_{c}(M)$, which is dense in $B^{\alp}_{p, c(\bbN)}(M)$ by definition. \qedhere
\end{proof}

\begin{remark} 
From Lemma \ref{lem:cNBesov}, it follows by the dominated convergence theorem that if $u^{\ell}(t)$ belongs to $L^{3}_{t}(I; B^{\alp}_{3, c(\bbN)}(M))$ ($0 < \alp < 1$), then $s^{\frac{1-\alp}{2}} \nrm{\nb U(s)}_{L^{3}(I \times M)} \to 0$ as $s \to 0$.

\end{remark}

The following lemma gives the parabolic PDE satisfied by $W^{\ell}$.
\begin{lemma} \label{lem:eq4W}
Let $W^{\ell}$ be defined as above and $U^{\ell}$ a solution to $(\rd_{s} - \Hlap) U^{\ell} = 0$. Then $W^{\ell}$ satisfies (in the sense of distributions)
\begin{equation} \label{eq:eq4W}
(\rd_{s} - \Hlap )W^{\ell}
 = 	2 \nb_{j} (\nb_{k} U^{j} \nb^{k} U^{\ell})  
 	+ \tensor{(\calR_{-})}{^{\ell}_{m jk}} \nb^{m} (U^{j} U^{k}) 
	 + \nb^{m} \big[ \tensor{(\calR_{+})}{^{\ell}_{m jk}} U^{j} U^{k} \big],
\end{equation}
where $\tensor{(\calR_{\pm})}{^{\ell}_{m j k}} := \tensor{\Riem}{_{m j}^{\ell}_{k}} \pm g_{m j} \tensor{\Ric}{^{\ell}_{k}}$
\end{lemma}
\begin{proof} 
Computation.
\end{proof}

In the following lemma, we justify the use of Duhamel's principle in our situation, which is necessary in view of the well-known non-uniqueness for parabolic equations. The key ingredient is the uniqueness of the Hodge heat flow in $L^{2}(M)$, i.e., (1) in Lemma \ref{lem:props4heatFlow}.

\begin{lemma} \label{lem:Duhamel}
Denote the right-hand side of \eqref{eq:eq4W} by $\calN^{\ell}(t, s)$. For $t \in I$ and $s \in (0,1]$, we have
\begin{equation} \label{eq:Duhamel}
	W^{\ell}(t, s) = \lim_{\eps \to 0} \int_{\eps}^{s} e^{(s-s') \Hlap} \calN^{\ell}(t, s') \, \ud s'.
\end{equation}
where the limit on the right-hand side is in the sense of distributions.
\end{lemma}

\begin{proof} 
In what follows, we omit writing $t$. We begin by observing that $\int_{\eps}^{s} e^{(s-s') \Hlap} \calN^{\ell}(s') \, \ud s'$ can be put in $C_{t}((0, 1]; L^{2}(M))$, thanks to Lemma \ref{lem:props4heatFlow} (in particular \eqref{eq:L2smth}) and Sobolev. 
We claim that for every $\eps > 0$,
\begin{equation} \label{eq:Duhamel:pf}
	e^{\eps \Hlap} \nb_{j} (u^{j} u^{\ell}) \in L^{2}(M).
\end{equation}

Indeed, for every $\omg_{\ell} \in C^{\infty}_{c}(M)$, by \eqref{eq:L2smth} and Sobolev, we have
\begin{align*}
	\abs{\int_{M} e^{\eps \Hlap} \nb_{j} (u^{j} u^{\ell}) \omg_{\ell} \, \ud^{d} \V}
	=& \abs{\int_{M}  (u^{j} u^{\ell}) \nb_{j} e^{\eps \Hlap} \omg_{\ell} \, \ud^{d} \V} \\
	\leq& C_{\eps} \nrm{u}^{2}_{L^{2}(M)} \nrm{\omg}_{L^{2}(M)},
\end{align*}

It follows that $W^{\ell} \in C_{t}((0, 1], L^{2}(M))$ as well, again using \eqref{eq:L2smth} and Sobolev. Thus, by the uniqueness statement in Lemma \ref{lem:props4heatFlow} and Duhamel's principle, for every $0 < \eps < s < 1$ we have 
\begin{equation*}
W^{\ell}(s) = e^{(s-\eps) \Hlap} W^{\ell}(\eps) + \int_{\eps}^{s} e^{(s-s') \Hlap} \calN^{\ell}(s') \, \ud s'.
\end{equation*}

To prove \eqref{eq:Duhamel}, we need to show $e^{(s-\eps) \Hlap} W^{\ell}(\eps) \to 0$ in the sense of distributions as $\eps \to 0$. 
For $\omg_{\ell} \in C^{\infty}_{c}(M)$, we write
\begin{align*}
	\int_{M} & \omg_{\ell} e^{(s-\eps) \Hlap} W^{\ell}(\eps)  \, \ud^{d} \V \\
	=& - \int_{M} \nb_{j} e^{s \Hlap} \omg_{\ell} \big[ u^{j} u^{\ell} - U^{j}(\eps) U^{\ell}(\eps) \big] \, \ud \V
		- \int_{M} \nb_{j} \big[ (e^{s \Hlap} - e^{(s-\eps) \Hlap}) \omg_{\ell} \big] U^{j}(\eps) U^{\ell}(\eps) \, \ud \V
\end{align*}

The first integral goes to zero as $\eps \to 0$, since $u^{j} u^{\ell} - U^{j}(\eps) U^{\ell}(\eps) \to 0$ in $L^{1}(M)$ by Lemma \ref{lem:props4heatFlow}, and $\nb_{j} e^{s \Hlap} \omg_{\ell} \in C^{0}(M)$ by \eqref{eq:L2smth}.  Moreover, from the $L^2$ convergence of $ (e^{\eps \Hlap} - 1)\omg_{\ell} \to 0$, we can see that $\nb_{j} [ (e^{s \Hlap} - e^{(s-\eps) \Hlap}) \omg_{\ell}] = \nb_{j}[e^{(s - \eps)\Hlap} (e^{\eps \Hlap} - 1)\omg_{\ell}] \to 0$ in $C^{0}(M)$ by Sobolev and \eqref{eq:L2smth} for high derivatives.  Since $U^{j}(\eps) U^{\ell}(\eps)$ remains uniformly bounded in $L^{1}(M)$ by Lemma \ref{lem:props4heatFlow}, we conclude that the second integral vanishes as $\eps \to 0$ as well. \qedhere

\end{proof}

From Lemma \ref{lem:Duhamel}, we conclude that $W^{\ell}(s) = W_{1}^{\ell}(s) + W_{2}^{\ell}(s) + W_{3}^{\ell}(s)$, with
\begin{align*}
W_{1}^{\ell}(s) =&	2 \int_{0}^{s} e^{(s-s') \Hlap} \nb_{j} (\nb_{k} U^{j} (s') \nb^{k} U^{\ell} (s')) \, \ud s' \\  
W_{2}^{\ell}(s) = & 	\int_{0}^{s} e^{(s-s') \Hlap} \tensor{(\calR_{-})}{^{\ell}_{m j k}} \nb^{m} (U^{j}(s') U^{k}(s')) \, \ud s' \\
W_{3}^{\ell}(s) = &	 \int_{0}^{s} e^{(s-s') \Hlap} \nb^{m} \bb( \tensor{(\calR_{+})}{^{\ell}_{m jk}} U^{j}(s') U^{k}(s') \bb) \, \ud s'
\end{align*}
where the integral $\int_{0}^{s} \, \ud s'$ must be interpreted as in Lemma \ref{lem:Duhamel}.  We remark that $W^{\ell}_{1}$ is the main contribution with a structure similar to those in \cite[Eq. (10)]{MR1298949} and \cite[Eq. (20)]{MR2422377}, and is the only term present in the flat case.

We are now ready to prove Theorem \ref{thm:commEst}.

\begin{proof} [Proof of Theorem \ref{thm:commEst}]
Below, all space-time integrals are over $I \times M$. We omit writing $t$. 

Let $s \in (0, 1]$. 
Taking out the $s'$-integral and integrating by parts, we have
\begin{align*}
\iint \eta \, W_{1}^{\ell}(s) U_{\ell}(s) \, \ud^{1+d} \V
= &  2 \int_{0}^{s} \iint \eta \, e^{(s-s') \Hlap} \nb_{j} (\nb_{k} U^{j} (s') \nb^{k} U^{\ell} (s')) U_{\ell}(s) \, \ud^{1+d} \V \, \ud s' \\
= &  - 2 \int_{0}^{s} \iint \eta \, \nb_{k} U^{j} (s') \nb^{k} U^{\ell} (s') \nb_{j} U_{\ell}(2s-s') \, \ud^{1+d} \V \, \ud s' 
\end{align*}

The $s'$ integrand for every $s' \in (0, s]$ is estimated using H\"older by
\begin{equation} \label{eq:commEst:pf:0}
	\leq C (s')^{-1+\alp} (2s-s')^{-(1-\alp)/2} \, \calU(s')^{2} \calU(2s-s')
\end{equation}
where $\calU(s) := s^{\frac{1-\alp}{2}} \nrm{\nb U(s)}_{L^{3}(I \times M)}$.  Observe that $\calU(s) \leq C \nrm{u}_{L^{3}_{t} (I; B^{\alp}_{3, \infty}(M))} < \infty$ for $0 < s \leq 1$.

For $\alp > 0$, \eqref{eq:commEst:pf:0} is integrable on $(0, s]$, and after rescaling the $ds'$ integral we obtain the estimate
\begin{align*}
\left\abs{\iint \eta \, W_{1}^{\ell}(s) U_{\ell}(s) \, \ud^{1+d} \V \right} \leq &C s^{-\frac{1 - 3 \alp}{2}} \int_0^1 \sigma^{-1+\alp} (2 - \sigma)^{-(1-\alp)/2}\calU(s\sigma)^{2} \calU(s(2-\sigma)) ~d\sigma
\end{align*}
The power of $s$ is non-negative since $\alp \geq 1/3$, and from the Remark after Lemma \ref{lem:cNBesov}, we furthermore have $\calU(s) \to 0$ as $s \to 0$.  Applying the dominated convergence theorem, it follows that
\begin{equation} \label{eq:commEst:pf:1}
\left\abs{\iint \eta \, W_{1}^{\ell}(s) U_{\ell}(s) \, \ud^{1+d} \V \right} \to 0 \hbox{ as } s \to 0.
\end{equation}
 

The other two terms are handled in the same way; in fact, they obey more favorable estimates. Proceeding similarly as before, we see that
\begin{align*}
\iint \eta \, W_{2}^{\ell}(s) U_{\ell}(s) \, \ud^{1+d} \V
=& \int_{0}^{s} \iint \eta \tensor{(\calR_{-})}{^{\ell}_{m j k}} \nb^{m} U^{j}(s') U^{k}(s') U_{\ell}(2s-s') \, \ud^{1+d} \V \, \ud s' \\
& + \int_{0}^{s} \iint \eta \tensor{(\calR_{-})}{^{\ell}_{m j k}}  U^{j}(s') \nb^{m} U^{k}(s') U_{\ell}(2s-s') \, \ud^{1+d} \V \, \ud s'.
\end{align*}

For $s' \in (0, s]$, the integrand of both $s'$-integrals are estimated using H\"older by
\begin{align*}
	\leq & C (s')^{-(1-\alp)/2} \nrm{\Riem}_{L^{\infty}(M)} \nrm{u}^{3}_{L^{3}_{t} (I; B^{\alp}_{3, \infty}(M))}.
\end{align*}

This is again integrable, provided $\alp > -1$. Thus upon integration, we obtain (as $\alp \geq 1/3$)
\begin{equation} \label{eq:commEst:pf:2}
\left\abs{\iint \eta \,  W_{2}^{\ell}(s) U_{\ell}(s) \, \ud^{1+d} \V\right} \leq C s^{(1+\alp)/2} \nrm{\Riem}_{L^{\infty}(M)} \nrm{u}^{3}_{L^{3}_{t} (I; B^{\alp}_{3, \infty}(M))} \to 0 \hbox{ as } s \to 0.
\end{equation}

For $W_{3}^{\ell}$, we compute 
\begin{align*}
\iint \eta \, W_{3}^{\ell}(s) U_{\ell}(s) \, \ud^{1+d} \V
=& - \int_{0}^{s} \iint \eta \tensor{(\calR_{+})}{^{\ell}_{m jk}} U^{j}(s') U^{k}(s') \nb^{m} U_{\ell}(2s-s') \, \ud^{1+d} \V \, \ud s',
\end{align*}
and as before, we then conclude (again using $\alp \geq 1/3$)
\begin{equation} \label{eq:commEst:pf:3}
\left\abs{\iint \eta \, W_{3}^{\ell}(s) U_{\ell}(s) \, \ud^{1+d} \V \right} 
\leq C s^{(1+\alp)/2} \nrm{\Riem}_{L^{\infty}(M)} \nrm{u}^{3}_{L^{3}_{t} (I; B^{\alp}_{3, \infty}(M))} \to 0 \hbox{ as } s \to 0.
\end{equation}

Combining \eqref{eq:commEst:pf:1}--\eqref{eq:commEst:pf:3}, we obtain \eqref{eq:commEst}. \qedhere
\end{proof}

\appendix
\section{$L^{p}$ theory of the Hodge heat flow on $M$} \label{sec:heatFlowEst}
In this appendix, we outline an approach for proving short-time $L^{p}$ estimates for solutions to \eqref{eq:HHF}; as an application, we give a proof of Proposition \ref{prop:embed}. Our method is elementary and self-contained. 

Let $u^{\ell} \in C^{\infty}_{c}(M)$, and $U^{\ell}(s) := e^{s \Hlap} u^{\ell}$ for $s > 0$. For $2 \leq p < \infty$ and $s \in (0, 1]$, we aim to show
\begin{align}
\nrm{U(s)}_{L^{p}(M)} \leq & C \nrm{u}_{L^{p}(M)} \label{eq:heatFlowEst:Lp:1} \\
\nrm{\nb U(s)}_{L^{p}(M)} \leq & C (\nrm{\nb u}_{L^{p}(M)} + \nrm{u}_{L^{p}(M)} ) \label{eq:heatFlowEst:Lp:2}\\
\nrm{\nb U(s)}_{L^{p}(M)} \leq & C s^{-1/2} \nrm{u}_{L^{p}(M)} \label{eq:heatFlowEst:Lp:3}
\end{align}
where $C > 0$ depends only on $p$ and $\nrm{\Riem}_{L^{\infty}(M)}$.

We now sketch a proof of the short-time $L^{p}$ estimates \eqref{eq:heatFlowEst:Lp:1}--\eqref{eq:heatFlowEst:Lp:3}. We first prove these estimates for $p = 2Z$, where $Z \geq 2$ is an integer. In fact, our proof below requires the Riemannian manifold $M$ to be merely complete with sectional curvatures bounded from above and below, i.e., $\nrm{\Riem}_{L^{\infty}(M)} < \infty$. The general case of $2 \leq p < \infty$ then follows by interpolating with the corresponding $L^2$-type estimates, which are easier. The relevant interpolation theory can be found in \cite[Chapter 7]{MR1163193} when $M$ has a `bounded geometry' (the definition can be found in \cite[Chapter 7]{MR1163193}), and can be deduced from the case of $M =\bbT^{d}$ when $M$ is compact, as in the proof of Proposition \ref{prop:embed} below.

We start with the following Bochner-type formula:
\begin{equation} \label{eq:Bochner:1}
\rd_{s} \abs{U}^{2} - \lap \abs{U}^{2} + 2 \abs{\nb U}^{2} = - 2 \Ric_{jk} U^{j} U^{k}
\end{equation}

Multiplying by $\abs{U}^{2(Z-1)}$ and integrating over $M$, we have
\begin{equation*} 
\frac{1}{Z} \rd_{s} \int_{M} \abs{U}^{2Z} \, \ud^{d} \V 
= \int_{M} (\lap \abs{U}^{2}  - 2 \abs{\nb U}^{2} - 2 \Ric_{jk} U^{j} U^{k} ) \abs{U}^{2(Z-1)} \, \ud^{d} \V.
\end{equation*}

The contribution of $-2\abs{\nb U}^{2}$ is non-positive. Integrating the $\lap$-term by parts\footnote{We remark that these formal manipulations may be made rigorous by using $u^{\ell} \in C^{\infty}_{c}(M)$ and $L^{2}(M)$ regularity theory for $U^{\ell}(s) = e^{s \Hlap} u^{\ell}$.}, we also have
\begin{equation*}
\int_{M} \lap \abs{U}^{2} \abs{U}^{2(Z-1)} \, \ud^{d} \V
= - (Z-1) \int_{M} \nb_{j} \abs{U}^{2} \nb^{j} \abs{U}^{2} \abs{U}^{2(Z-2)} \, \ud^{d} \V \leq 0.
\end{equation*}

Thus we arrive at
\begin{align*}
\frac{1}{Z} \rd_{s} \int_{M} \abs{U}^{2Z} \, \ud^{d} \V 
& \leq 2 \nrm{\Ric}_{L^{\infty}(M)} \int_{M}  \abs{U}^{2Z} \, \ud^{d} \V.
\end{align*}

Applying Gronwall, we conclude \eqref{eq:heatFlowEst:Lp:1}. 

To proceed, we need to compute the equation satisfied by $\nb_{j} u^{\ell}$. It is not difficult to check the schematic identity
\begin{equation*}
(\rd_{s} - \Hlap) \nb_{j} u^{\ell} = \Riem \cdot \nb u + \nb \Riem \cdot u,
\end{equation*}
where $\Riem \cdot \nb u$ is a linear combination of terms involving $\Riem$ and $\nb u$, and similarly for $\nb \Riem \cdot u$. Accordingly, we have the following schematic Bochner-type identity:
\begin{equation} \label{eq:Bochner:2}
	\rd_{s} \abs{\nb U}^{2} - \lap \abs{\nb U}^{2}  + 2 \abs{\nb \nb U}^{2} = \Riem \cdot \nb U \cdot \nb U + \nb \Riem \cdot U \cdot \nb U.
\end{equation}

As before, multiply by $\abs{\nb U}^{2(Z-1)}$ and integrate over $M$. The contribution of the first term on the right-hand is bounded by $\leq C\nrm{\Riem}_{L^{\infty}(M)} \int_{M} \abs{\nb U}^{2Z}$. For the second term, we will integrate by parts and apply Young's inequality to estimate
\begin{align*}
	& \int_{M} \nb \Riem \cdot U \cdot \nb U \abs{\nb U}^{2(Z-1)} \, \ud^{d} \V \\
	& \quad \leq 2 \int_{M} \abs{\nb \nb U}^{2} \abs{\nb U}^{2(Z-1)} \, \ud^{d} \V
	+ C_{Z} \int_{M} \nrm{\Riem}_{L^{\infty}(M)} \abs{\nb U}^{2Z}  + \nrm{\Riem}^{Z+1}_{L^{\infty}(M)} \abs{U}^{2Z} \, \ud^{d} \V
\end{align*}

Note that the first term on the last line cancels with the contribution of $2 \abs{\nb \nb U}^{2}$ in \eqref{eq:Bochner:2}. Thus, proceeding as before using Gronwall, as well as using \eqref{eq:heatFlowEst:Lp:1} to control $\int_{M} \abs{U}^{2Z}$, we obtain \eqref{eq:heatFlowEst:Lp:2}.

We now turn to \eqref{eq:heatFlowEst:Lp:3}. From \eqref{eq:Bochner:2}, we obtain
\begin{equation} \label{eq:Bochner:3}
	\rd_{s} (s \abs{\nb U}^{2}) - \lap (s \abs{\nb U}^{2}) + 2 s \abs{\nb \nb U}^{2} - \abs{\nb U}^{2} = s \Riem \cdot \nb U \cdot \nb U + s \nb \Riem \cdot U \cdot \nb U
\end{equation}

Adding one half of \eqref{eq:Bochner:1}, note that the undesirable last term on the left-hand side cancels, and we arrive at 
\begin{equation} \label{eq:Bochner:4}
	\rd_{s} \Psi + \lap \Psi + 2 s \abs{\nb \nb U}^{2} = s \Riem \cdot \nb U \cdot \nb U + \Riem \cdot U \cdot U + s \nb \Riem \cdot U \cdot \nb U.
\end{equation}
where $\Psi := s \abs{\nb U}^{2} + \frac{1}{2} \abs{U}^{2}$. We now multiply this equation by $\Psi^{Z-1}$ and integrate over $M$. The contribution of the first two terms of the right-hand side is bounded by 
\begin{equation*}
\leq C \nrm{\Riem}_{L^{\infty}(M)} \int_{M} \Psi^{Z} \, \ud^{d} \V
\end{equation*} 

For the last term of \eqref{eq:Bochner:4}, we integrate by parts and use Young's inequality to bound
\begin{align*}
	& \int_{M} s \nb \Riem \cdot U \cdot \nb U \cdot \Psi^{Z-1} \, \ud^{d} \V \\
	& \quad \leq \int_{M} 2 s \abs{\nb \nb U}^{2} \Psi^{Z-1} \, \ud^{d} \V
		+ C_{Z}  \int_{M} \nrm{\Riem}_{L^{\infty}(M)}\Psi^{Z} + s \nrm{\Riem}^{Z+1}_{L^{\infty}(M)} \abs{U}^{2Z} \, \ud^{d} \V
\end{align*}
and the first term on the last line again cancels with the contribution of $2 s \abs{\nb \nb U}^{2}$. Thus,
\begin{equation*}
	\rd_{s} \int_{M} \Psi^{Z} \ud^{d} \V \leq C_{Z, \nrm{\Riem}_{L^{\infty}(M)}} \int \Psi^{Z} \, \ud \V.
\end{equation*}

By Gronwall, we obtain $\int_{M} \Psi^{Z}(s) \leq C \int_{M} \Psi^{Z}(0)$ for $0 < s \leq 1$. Since $s \abs{\nb U(s)}^{2} \leq \Psi(s)$ and $\Psi(0) = \frac{1}{2}\abs{U(0)}^{2}$, we obtain \eqref{eq:heatFlowEst:Lp:3}.

We now give a proof of Proposition \ref{prop:embed}.
\begin{proof} [Proof of Proposition \ref{prop:embed}]
First, we prove (1). Let $0 < \alp \leq 1$ and $2 \leq p < \infty$. Recall that on any compact manifold $M$, $C^{\infty}_{c}(M)$ is dense in $W^{1, p}(M)$ \cite{MR1636569, MR1688256}. Furthermore, it follows from a general result in the theory of complex interpolation \cite[Theorem 4.2.2]{MR0482275} that $W^{1, p}(M)$ is densely embedded into every complex interpolation space $W^{\alp, p}(M)$ for $0 < \alp <1$. We therefore conclude that $C^{\infty}_{c}(M)$ is dense in $W^{\alp, p}(M)$. Thus, the assertion $W^{\alp, p}(M) \seq B^{\alp}_{p, c(\bbN)}(M)$ is equivalent to
\begin{equation*}
	s^{\frac{1-\alp}{2}} \nrm{\nb e^{s \Hlap} u}_{L^{p}(M)} 
	\leq C \nrm{u}_{W^{\alp, p}(M)}
\end{equation*}
for $u^{\ell} \in C^{\infty}_{c}(M)$ and $s \in (0, 1]$. But this follows from complex interpolation between \eqref{eq:heatFlowEst:Lp:2} and \eqref{eq:heatFlowEst:Lp:3}. 

Next, we prove (2). Let $0 < \alp < \alp' < 1$ and $1 \leq p < \infty$. We begin by noting that when $M = \bbT^{d}$, the inequality
\begin{equation} \label{eq:embed:torus}
	\nrm{u}_{W^{\alp, p}(\bbT^{d})} 
	\leq C\nrm{u}_{C^{\alp'}(\bbT^{d})}
\end{equation}
holds for every $u^{\ell} \in C^{\alp'}(\bbT^{d})$. This can be proved by, e.g., using the characterization of these spaces as in \cite[Theorems 2.3.1 and 2.4.1]{MR1163193}\footnote{On $\bbT^{d}$, we use the component-wise definition for spaces of vector fields. We also note for the reader that $C^{\alp'}(\bbT^{d}) = B^{\alp'}_{\infty, \infty}$ and $W^{\alp, p}(\bbT^{d}) = F^{\alp}_{p, 2}$ in the notations of \cite{MR1163193}.}; we omit the proof. In the general case, our idea is to reduce to the special case $M = \bbT^{d}$ by using a finite partition of unity, which exists thanks to compactness of $M$.

We now make our idea precise. Since $M$ is compact, for every $\dlt > 0$, there exists a locally finite covering $\set{P_{j}}_{j=1}^{J}$ of $M$ by geodesic balls of radius $\dlt > 0$, and also a partition of unity $\set{\psi_{j}}_{j=1}^{J}$ such that $\supp \psi_{j} \seq P_{j}$. Take $\dlt$ to be sufficiently small, so that each $P_{j}$ is contained in a single coordinate chart. Then for each $j = 1,\ldots, J$,  we can find a diffeomorphism $\Phi_{j} : P_{j} \to \Phi_{j}(P_{j}) \seq \bbT^{d}$. 

For $X(M) = C^{\alp}(M), L^{p}(M)$ or $W^{1, p}(M)$ (with $0 < \alp < 1$, $1 \leq p \leq \infty$), by its localization and diffeomorphism invariance properties, there exist constants $c, C > 0$ such that
\begin{equation} \label{eq:embed:equivNorms}
	c \nrm{\psi_{j} u}_{X(M)} \leq \sum_{j=1}^{J} \nrm{(\Phi_{P_{j}})_{\star} (\psi_{j} u) }_{X(\bbT^{d})} \leq C \nrm{\psi_{j} u}_{X(M)} 
\end{equation}

for every $u^{\ell} \in X(M)$, where $(\Phi_{P_{j}})_{\star} (\psi_{j} u)$ is the pushforward of $\psi_{j}u^{\ell}$ to a vector field on $\Phi_{j}(P_{j}) \seq \bbT^{d}$. By complex interpolation, \eqref{eq:embed:equivNorms} also holds for $X(M) = W^{\alp, p}(M)$ with $0 < \alp < 1$. Now the desired inclusion $C^{\alp'}(M) \seq W^{\alp, p}(M)$ follows from \eqref{eq:embed:torus}. \qedhere
\end{proof}

\section{Equivalence of Besov-type spaces when $M = \bbT^{d}$ or $\bbR^{d}$} \label{sec:equivBesov}
In this appendix, we outline a proof of the equivalence of the Besov-type spaces as in Definition \ref{def:besov} and the standard spaces defined by Littlewood-Paley theory on $\bbT^{d}$ and $\bbR^{d}$.

We begin by recalling the definition of (inhomogeneous) Littlewood-Paley projections; we borrow the notations from \cite{MR2422377}.
Let $\chi : [0, \infty) \to \bbR$ be a smooth non-negative function such that $\chi(\rho) = 1$ for $\rho \leq 1/2$ and $\chi(\rho) = 0$ for $\rho \geq 1$. Let us denote by $\calF$ the Fourier transform on $\bbT^{d}$ or $\bbR^{d}$.
For $k \geq 0$ an integer and a scalar function $u$, define
\begin{equation*}
	\Dlt_{k} u := \calF^{-1} \bb[ \bb( \chi(\abs{\xi}/2^{k+1}) - \chi(\abs{\xi}/2^{k}) \bb) \calF u(\xi) \bb], \quad
	\Dlt_{-1} u := \calF^{-1} (\chi(\abs{\xi}) \calF u ).
\end{equation*}

We now define the standard Besov spaces on $\bbT^{d}$ and $\bbR^{d}$; to distinguish from those in Definition \ref{def:besov}, we shall accent these spaces with a hat, e.g., $\widehat{B}^{\alp}_{p, r}$.

For $\alp \in \bbR$ and $1 \leq p, r \leq \infty$, the standard definition of the inhomogeneous Besov norm $\nrm{\cdot}_{\widehat{B}^{\alp}_{p, r}}$ (via Littlewood-Paley theory) reads 
\begin{equation*}
	\nrm{u}_{\widehat{B}^{\alp}_{p, r}} := \nrm{\Dlt_{-1} u}_{L^{p}} + \bb( \sum_{k \geq 0} (2^{\alp k}\nrm{\Dlt_{k} u}_{L^{p}})^{r} \bb)^{1/r},
\end{equation*}
with the usual modification for $r = \infty$. The space $\widehat{B}^{\alp}_{p, r}$ is defined to be the space of tempered distributions $u$ such that $\nrm{u}_{\widehat{B}^{\alp}_{p, r}} < \infty$. 
Following \cite{MR2422377}, for $\alp \in \bbR$ and $1 \leq p \leq \infty$ we also define the space $\widehat{B}^{\alp}_{p, c(\bbN)}$ as
\begin{equation*}
	\widehat{B}^{\alp}_{p, c(\bbN)} := \set{u \in \widehat{B}^{\alp}_{p, \infty} : 2^{\alp k} \nrm{\Dlt_{k} u}_{L^{p}} \to 0 \hbox{ as } k \to \infty}.
\end{equation*}

We extend these spaces component-wise to vector fields and 1-forms. 

For $\alp > 0$ and $1 \leq p, r \leq \infty$, note that $\widehat{B}^{\alp}_{p, r}, \widehat{B}^{\alp}_{p, c(\bbN)} \seq L^{p}$. Moreover, it is a standard fact that for $\alp > 0$ and $1 \leq p, r < \infty$, $C^{\infty}_{c}$ is dense in $\widehat{B}^{\alp}_{p, r}$. On the other hand, for $\alp > 0$, $1 \leq p < \infty$ and $r = \infty$, the closure of $C^{\infty}_{c}$ under the norm $\nrm{\cdot}_{\widehat{B}^{\alp}_{p, \infty}}$ is not $\widehat{B}^{\alp}_{p, \infty}$, but $\widehat{B}^{\alp}_{p, c(\bbN)}$.

The main result of this appendix is the following.
\begin{proposition} 
Let $M = \bbT^{d}$ or $\bbR^{d}$. Then for $0 < \alp < 1$, $1 \leq p, r < \infty$, we have
\begin{equation*}
	B^{\alp}_{p, r}(M) = \widehat{B}^{\alp}_{p, r}, \quad
	B^{\alp}_{p, c(\bbN)}(M) = \widehat{B}^{\alp}_{p, c(\bbN)}.
\end{equation*}

More precisely, the two spaces coincide and the norms are equivalent.
\end{proposition}

\begin{proof} 
In view of density of $C^{\infty}_{c}$ in all of these spaces, it suffices to establish equivalence of the norms $\nrm{\cdot}_{B^{\alp}_{p, r}(M)}$ and $\nrm{\cdot}_{\widehat{B}^{\alp}_{p, r}}$ for $1 \leq p < \infty$, $1 \leq r \leq \infty$.
The key estimates are provided by the lemma below; note that on $\bbT^{d}$ or $\bbR^{d}$, the Hodge heat semi-group $e^{s \Hlap}$ just becomes the (component-wise) standard heat semi-group $e^{s \lap}$.
\begin{lemma} \label{lem:equivBesov}
Let $1 \leq p < \infty$ and $\alp \geq 0$. Then for $k \geq -1$ an integer and $0 < s \leq 1$, we have
\begin{align} 
2^{\alp k} s \nrm{\Dlt_{k} \rd_{s} e^{s \lap} u}_{L^{p}} 
\leq C & \min\set{(s^{\frac{1}{2}} 2^{k})^{\alp}, (s^{\frac{1}{2}} 2^{k})^{\alp - 1}} s^{\frac{1-\alp}{2}} \nrm{\nb e^{\frac{s}{2} \lap} u}_{L^{p}}  \label{eq:equivBesov:1} \\
s^{\frac{1-\alp}{2}} \nrm{\nb e^{s \lap} \Dlt_{k} u}_{L^{p}} 
\leq C & \min\set{(s^{\frac{1}{2}} 2^{k})^{1- \alp},(s^{\frac{1}{2}} 2^{k})^{-\alp}} (2^{\alp k} \nrm{\Dlt_{k} u}_{L^{p}}) \label{eq:equivBesov:2}
\end{align}
\end{lemma}

Assuming the lemma for the moment, we sketch how equivalence of the norms $\nrm{\cdot}_{B^{\alp}_{p, r}(M)}$, $\nrm{\cdot}_{\widehat{B}^{\alp}_{p, r}}$ is proven. To establish $\nrm{\cdot}_{B^{\alp}_{p, r}(M)} \leq C \nrm{\cdot}_{\widehat{B}^{\alp}_{p, r}}$, we begin by writing
\begin{align*}
	\nrm{u}_{B^{\alp}_{p, r}(M)} 
	=& \nrm{u}_{L^{p}} + \bb\Vert s^{\frac{1-\alp}{2}} \nrm{\sum_{k \geq -1} \nb e^{s \lap} \Dlt_{k} u }_{L^{p}} \bb\Vert_{L^{r}( (0, 1], \ud s/ s )}.
\end{align*}

The first term is bounded by $\nrm{u}_{\widehat{B}^{\alp}_{p, r}}$ as $\alp > 0$. For the second term, we apply triangle and \eqref{eq:equivBesov:2} in Lemma \ref{lem:equivBesov} to estimate it by 
\begin{equation} \label{eq:equivBesov:pf:1}
	\leq C \bb\Vert \sum_{k \geq -1}  K_{1}(s, k) (2^{\alp k} \nrm{\Dlt_{k} u}_{L^{p}}) \bb\Vert_{L^{r}( (0, 1], \ud s/ s )}.
\end{equation}
where $K_{1}(s, k) := \min\set{(s^{\frac{1}{2}} 2^{k})^{1- \alp},(s^{\frac{1}{2}} 2^{k})^{-\alp}}$. As 
\begin{equation*}
\sup_{k \geq -1} \int_{0}^{1} K_{1}(s, k) \, \frac{\ud s}{s} < \infty , \quad
\sup_{s \in (0, 1]} \sum_{k \geq -1} K_{1}(s,k) < \infty
\end{equation*} 
it follows by Schur's test that $\eqref{eq:equivBesov:pf:1} \leq C \nrm{2^{\alp k} \nrm{\Dlt_{k} u}_{L^{p}}}_{\ell^{r}(\set{k \geq -1})} \leq C \nrm{u}_{B^{\alp}_{p, r}}$ as desired.

To prove the other direction $\nrm{\cdot}_{\widehat{B}^{\alp}_{p, r}} \leq C \nrm{\cdot}_{B^{\alp}_{p, r}(M)}$, we write
\begin{align*}
	\nrm{u}_{\widehat{B}^{\alp}_{p, r}} 
	=& \nrm{\Dlt_{-1} u}_{L^{p}} + \bb\Vert 2^{\alp k} \bb\Vert- \int_{0}^{1} s \Dlt_{k} \rd_{s} e^{s \lap}  u \, \frac{\ud s}{s} + \Dlt_{k} e^{\lap}  u \bb\Vert_{L^{p}} \bb\Vert_{\ell^{r}(\set{k \geq 0 })}.
\end{align*}

Note that $\nrm{\Dlt_{-1} u}_{L^{p}} \leq C \nrm{u}_{L^{p}}$. Moreover, 
\begin{equation*}
\nrm{2^{\alp k} \nrm{\Dlt_{k} e^{\lap} u}_{L^{p}} }_{\ell^{r}(\set{k \geq 0})} 
\leq C \nrm{2^{(\alp-1) k} \nrm{\abs{\nb} \, \Dlt_{k} e^{\lap} u}_{L^{p}}}_{\ell^{r}(\set{k \geq 0})} 
\leq C \nrm{u}_{L^{p}}
\end{equation*} 
where $\abs{\nb} = \sqrt{-\lap}$. Thus, it remains to show
\begin{equation*}
	\bb\Vert 2^{\alp k} \int_{0}^{1} s \nrm{\Dlt_{k} \rd_{s} e^{s \lap} u}_{L^{p}} \, \frac{\ud s}{s} \bb\Vert_{\ell^{r}(\set{k \geq 0})} 
	\leq C \nrm{u}_{B^{\alp}_{p, r}(M)}.
\end{equation*}

This can be proven similarly as before, using \eqref{eq:equivBesov:1} in Lemma \ref{lem:equivBesov} instead of \eqref{eq:equivBesov:2}. \qedhere
\end{proof}

\begin{proof} [Proof of Lemma \ref{lem:equivBesov}]
The guiding principle is that $e^{s \lap}$ behaves like a frequency cut-off at $\abs{\xi} \leq C s^{-\frac{1}{2}}$, and thus we obtain more favorable estimate by moving the derivative to fall on the lower frequency cut-off.
Since $\rd_{s} e^{s \lap} = \nb^{\ell} \nb_{\ell} e^{s \lap} = \nb^{\ell} e^{\frac{s}{2} \lap} \nb_{\ell} e^{\frac{s}{2} \lap}$, we have for $k \geq -1$ and $0 < s \leq 1$ (via standard Littlewood-Paley theory and heat kernel estimates) 
\begin{align*}
2^{\alp k} s \nrm{\Dlt_{k} \rd_{s} e^{s \lap} u}_{L^{p}} 
\leq C (s^{\frac{1}{2}} 2^{k})^{\alp}  s^{\frac{1-\alp}{2}} \nrm{\nb e^{\frac{s}{2} \lap} u}_{L^{p}}.
\end{align*} 

To finish the proof of \eqref{eq:equivBesov:1}, we need an improvement in the regime $k \geq 0$ and $s^{\frac{1}{2}} 2^{k} > 1$. Using $\nrm{\Dlt_{k} (\cdot) }_{L^{p}} \leq C 2^{-k} \nrm{\abs{\nb} \, \Dlt_{k} (\cdot) }_{L^{p}}$ and $\nrm{\abs{\nb} \, \nb_{\ell} e^{\frac{s}{2} \lap} (\cdot) }_{L^{p}} \leq C s^{-1} \nrm{\cdot}_{L^{p}}$, we get
\begin{equation*}
2^{\alp k} s \nrm{\Dlt_{k} \rd_{s} e^{s \lap} u}_{L^{p}} 
\leq C (s^{\frac{1}{2}} 2^{k})^{\alp-1}  s^{\frac{1-\alp}{2}} \nrm{\nb e^{\frac{s}{2} \lap} u}_{L^{p}}
\end{equation*}
which proves \eqref{eq:equivBesov:1}. 

The remaining estimate \eqref{eq:equivBesov:2} is proven similarly, moving the derivative $\nb$ to fall on $e^{s \lap}$ or $\Dlt_{k}$ depending on whether $s^{\frac{1}{2}} 2^{k} < 1$ or otherwise, respectively. \qedhere
\end{proof}


\bibliographystyle{amsplain}

\begin{thebibliography}{10}

\bibitem{MR1636569}
T. Aubin, \emph{Some nonlinear problems in {R}iemannian geometry},
  Springer Monographs in Mathematics, Springer-Verlag, Berlin, 1998.
  \MR{1636569 (99i:58001)}

\bibitem{MR0482275}
J. Bergh and J. L{\"o}fstr{\"o}m. \emph{Interpolation spaces. {A}n introduction}, {Grundlehren der Mathematischen Wissenschaften, No. 223}, Springer-Verlag, 1976.

\bibitem{Buckmaster:2013vv}
T. Buckmaster, \emph{{Onsager's conjecture almost everywhere in time}},
  arXiv.org (2013).

\bibitem{Buckmaster:2013vy}
T. Buckmaster, C. De~Lellis, and L. Sz{\'e}kelyhidi, Jr,
  \emph{{Transporting microstructure and dissipative Euler flows}}, arXiv.org
  (2013).

\bibitem{MR2422377}
A.~Cheskidov, P.~Constantin, S.~Friedlander, and R.~Shvydkoy, \emph{Energy
  conservation and {O}nsager's conjecture for the {E}uler equations},
  Nonlinearity \textbf{21} (2008), no.~6, 1233--1252. \MR{2422377
  (2009g:76008)}

\bibitem{MR1298949}
P. Constantin, W. E, and E.~S. Titi, \emph{Onsager's conjecture on
  the energy conservation for solutions of {E}uler's equation}, Comm. Math.
  Phys. \textbf{165} (1994), no.~1, 207--209. \MR{1298949 (96e:76025)}

\bibitem{choff}
A.~Choffrut.
\newblock h-{P}rinciples for the {I}ncompressible {E}uler {E}quations.
\newblock {\em Arch. Ration. Mech. Anal.}, 210(1):133--163, 2013.

\bibitem{deLSzeCts2d}
A.~Choffrut, C.~De~Lellis, and L.~Sz{\' e}kelyhidi, Jr.
\newblock Dissipative continuous {E}uler flows in two and three dimensions.
  {P}reprint.
\newblock 2012.

\bibitem{DeLellis:2009jh}
C. De~Lellis and L. Sz{\'e}kelyhidi, Jr, \emph{{The Euler
  equations as a differential inclusion}}, Annals of Mathematics (2009).

\bibitem{DeLellis:2012uza}
\bysame, \emph{{Dissipative continuous Euler flows}}, arXiv.org (2012).

\bibitem{DeLellis:2012tz}
\bysame, \emph{{Dissipative Euler Flows and Onsager's Conjecture}}, arXiv.org
  (2012).

\bibitem{MR1734632}
J. Duchon and R. Robert, \emph{Inertial energy dissipation for weak
  solutions of incompressible {E}uler and {N}avier-{S}tokes equations},
  Nonlinearity \textbf{13} (2000), no.~1, 249--255. \MR{1734632 (2001c:76032)}

\bibitem{Eyink1994222}
G. L. Eyink, \emph{Energy dissipation without viscosity in ideal
  hydrodynamics I. Fourier analysis and local energy transfer}, Physica D:
  Nonlinear Phenomena \textbf{78} (1994), no.~3‚{\"A}{\`\i}4, 222 -- 240.

\bibitem{MR1688256}
E. Hebey, \emph{Nonlinear analysis on manifolds: {S}obolev spaces and
  inequalities}, Courant Lecture Notes in Mathematics, vol.~5, New York
  University Courant Institute of Mathematical Sciences, New York, 1999.
  \MR{1688256 (2000e:58011)}

\bibitem{Isett:2012ub}
P. Isett, \emph{{H{\"o}lder Continuous Euler Flows in Three Dimensions with
  Compact Support in Time}}, arXiv.org (2012).


\bibitem{MR2221254}
S.~Klainerman and I.~Rodnianski, \emph{A geometric approach to the
  {L}ittlewood-{P}aley theory}, Geom. Funct. Anal. \textbf{16} (2006), no.~1,
  126--163. \MR{2221254 (2007e:58046)}

\bibitem{Onsager}
L. Onsager,
\emph{Statistical hydrodynamics}, Nuovo Cimento (Suppl.) \textbf{6} (1949) 279 -- 287

\bibitem{MR0042769}
A.~N. Milgram and P.~C. Rosenbloom, \emph{Harmonic forms and heat conduction.
  {I}. {C}losed {R}iemannian manifolds}, Proc. Nat. Acad. Sci. U. S. A.
  \textbf{37} (1951), 180--184. \MR{0042769 (13,160a)}

\bibitem{MR2660721}
R. Shvydkoy, \emph{Lectures on the {O}nsager conjecture}, Discrete Contin.
  Dyn. Syst. Ser. S \textbf{3} (2010), no.~3, 473--496. \MR{2660721
  (2011h:76051)}

\bibitem{MR0252961}
E.~M. Stein, \emph{Topics in harmonic analysis related to the
  {L}ittlewood-{P}aley theory.}, Annals of Mathematics Studies, No. 63,
  Princeton University Press, Princeton, N.J., 1970. \MR{0252961 (40 \#6176)}
  
\bibitem{Strichartz:1983uw}
R. Strichartz, \emph{{Analysis of the Laplacian on the complete Riemannian
  manifold}}, J. Funct. Anal (1983).

\bibitem{MR1163193}
H. Triebel, \emph{{Theory of function spaces. {II}}}, Monographs in Mathematics, Birkh\"auser Verlag, 1992.

\end{thebibliography}

\providecommand{\bysame}{\leavevmode\hbox to3em{\hrulefill}\thinspace}
\providecommand{\MR}{\relax\ifhmode\unskip\space\fi MR }
\providecommand{\MRhref}[2]{%
  \href{http://www.ams.org/mathscinet-getitem?mr=#1}{#2}
}
\providecommand{\href}[2]{#2}


\end{document}